\documentclass{amsart}
\usepackage{amssymb,latexsym}
\usepackage{amscd,amsthm}

\usepackage[all]{xy}
\usepackage{tikz}

\usepackage{tikz-cd}

\newtheorem{theorem}{Theorem}[section]

\newtheorem{lemma}[theorem]{Lemma}
\newtheorem{proposition}[theorem]{Proposition}
\newtheorem{corollary}[theorem]{Corollary}

\theoremstyle{definition}
\newtheorem{definition}[theorem]{Definition}

\newtheorem*{remark}{Remark}

\newtheorem{example}[theorem]{Example}

\newtheorem*{setup}{Set up}

\DeclareMathOperator{\Ext}{Ext}
\DeclareMathOperator{\Hom}{Hom}
\DeclareMathOperator{\Tor}{Tor}

\newcommand{\cat}[1]{\mathcal{#1}}           

\newcommand{\tensor}{\otimes}

\newcommand{\class}[1]{\mathcal{#1}}   

\newcommand{\Z}{\mathbb{Z}}
\newcommand{\Q}{\mathbb{Q/Z}}

\newcommand{\ch}{\textnormal{Ch}(R)}

\newcommand{\tilclass}[1]{\widetilde{\class{#1}}}

\newcommand{\dwclass}[1]{dw\widetilde{\class{#1}}}

\newcommand{\rightperp}[1]{#1^{\perp}}
\newcommand{\leftperp}[1]{{}^\perp #1}

\begin{document}

\title{Duality pairs and stable module categories}

\author{James Gillespie}
\address{Ramapo College of New Jersey \\
         School of Theoretical and Applied Science \\
         505 Ramapo Valley Road \\
         Mahwah, NJ 07430}
\email[Jim Gillespie]{jgillesp@ramapo.edu}
\urladdr{http://pages.ramapo.edu/~jgillesp/}

\date{\today}

\begin{abstract}
Let $R$ be a commutative ring. We show that any complete duality pair gives rise to a theory of relative homological algebra, analogous to Gorenstein homological algebra. Indeed Gorenstein homological algebra over a commutative Noetherian ring of finite Krull dimension can be recovered from the duality pair $(\class{F},\class{I})$ where $\class{F}$ is the class of flat $R$-modules and $\class{I}$ is the class of injective $R$-modules. For a general $R$, the AC-Gorenstein homological algebra of Bravo-Gillespie-Hovey is the one coming from the duality pair $(\class{L},\class{A})$ where $\class{L}$ is the class of level $R$-modules and $\class{A}$ is class of absolutely clean $R$-modules.  Indeed we show here that the work in~\cite{bravo-gillespie-hovey} can be extended to obtain similar abelian model structures on $R$-Mod from any a complete duality pair $(\class{L},\class{A})$. It applies in particular to the original duality pairs constructed by Holm-J\o rgensen.
\end{abstract}

\maketitle

\section{Introduction}

The goal of this paper is to show how Gorenstein homological algebra can be done with respect to a duality pair. This comes from extracting and generalizing the core idea behind the main results in~\cite{bravo-gillespie-hovey}. 

Let $R$ be a commutative ring with identity.  The notion of a duality pair of $R$-modules was introduced by Holm and J\o rgensen in~\cite{holm-jorgensen-duality}. The notion is very useful in relative homological algebra because, as shown in~\cite[Theorem~3.1]{holm-jorgensen-duality}, duality pairs are very closely related to purity, existence of covers and envelopes, and to complete cotorsion pairs. Although the language did not quite agree, duality pairs were used by Bravo-Gillespie-Hovey in~\cite[Theorem~A.6]{bravo-gillespie-hovey} to prove the equivalence of two different types of acyclicity that can be associated to chain complexes of projective $R$-modules. In this paper we merge these two notions of duality pair, getting what we call a \emph{complete duality pair}, and show that an entire theory of Gorenstein homological algebra, complete with associated abelian model structures with stable homotopy categories, can be done with respect to a complete duality pair. Indeed the theory of AC-Gorenstein homological algebra introduced in~\cite{bravo-gillespie-hovey} is the special case of having the complete duality pair $(\class{L},\class{A})$ where $\class{L}$ is the class of \emph{level} $R$-modules and $\class{A}$ is the class of \emph{absolutely clean} $R$-modules. Although one could strive to get analogous results for noncommutive rings, the theory for commutative rings worked out in this paper is quite elegant and Holm and J\o rgensen, throughout~\cite{holm-jorgensen-duality}, provide several interesting examples of complete duality pairs over certain commutative Noetherian rings, for which our theorems here apply. See Section~\ref{Sec-duality pairs}.

Let us indicate the relevant definitions and then state more precisely the main results proved in this paper. First, recall that for a given $R$-module $M$, its \emph{character module} is defined to be the $R$-module $M^+ = \Hom_{\Z}(M,\Q)$. A duality pair in the sense of~\cite{holm-jorgensen-duality} is essentially a pair of classes $(\class{M},\class{C})$ such that $M \in \class{M}$ if and only if $M^+ \in \class{C}$; see Definition~\ref{def-duality pair} for the complete definition which involves further closure properties on the classes $\class{M}$ and $\class{C}$. The concept of duality pair used in~\cite{bravo-gillespie-hovey}, and in this paper, requires that $(\class{C},\class{M})$ be a duality pair too. This is incorporated into the definition of \emph{complete duality pair} given in this paper; see Definition~\ref{def-complete duality pair}. The examples of duality pairs given throughout~\cite{holm-jorgensen-duality, bravo-gillespie-hovey, bravo-perez} provide many interesting complete duality pairs; again see Section~\ref{sec-example duality pairs}.

\begin{definition}\label{Defs}
Assume $(\class{L},\class{A})$ is a complete duality pair. Let $M$ be an $R$-module.
\begin{enumerate}
\item  We say $M$ is \textbf{Gorenstein $(\class{L},\class{A})$-injective} if
$M=Z_{0}I$ for some exact complex of injectives $I$ for which $\Hom_R(A,I)$ remains exact for all $A \in \class{A}$. 

\item We say $M$ is \textbf{Gorenstein $(\class{L},\class{A})$-projective} if
$M=Z_{0}P$ for some exact complex of projectives $P$ for which $\Hom_R(P,L)$ remains exact for all $L \in \class{L}$.     

\item We say $M$ is \textbf{Gorenstein $(\class{L},\class{A})$-flat} if
$M=Z_{0}F$ for some exact complex of flat modules $F$ for which $A \otimes_R F$ remains exact for all $A \in \class{A}$.
\end{enumerate}
\end{definition}

Note that with these definitions, a Gorenstein $(\class{L},\class{A})$-projective module $M$ is always Gorenstein $(\class{L},\class{A})$-flat, by Theorem~\ref{them-projectivecomplexes}. We note too that when $R$ is a commutative Noetherian ring of finite Krull dimension, then these definitions, applied to the flat-injective duality pair, agree with the usual definitions of Gorenstein injective, projective and flat modules studied by Enochs and many other authors. See~\cite{enochs-jenda-book} for the basic theory of these modules.

 Recall that a model category is a category satisfying extra axioms that lead to a formal homotopy theory on that category. If the category is abelian this leads to some sort of (relative) homological algebra. For example, Hovey showed that the theory of Gorenstein homological algebra, over a Gorenstien ring $R$, is encoded within abelian model structures on the category $R$-Mod, of $R$-modules. This appeared in~\cite{hovey} where Hovey also made a general study of abelian model categories. One can also see~\cite{gillespie-hereditary-abelian-models} for a recent survey article on abelian model categories. Now, based on ideas from~\cite{bravo-gillespie-hovey} we are able to prove the following results in this paper.

\begin{theorem}\label{thm-main}
Assume $(\class{L},\class{A})$ is a complete duality pair. Let $R$-Mod denote the category of $R$-modules. 
\begin{enumerate}
\item There is a cofibrantly generated abelian model structure on $R$-Mod in which every object is cofibrant and the fibrant objects are the Gorenstein $(\class{L},\class{A})$-injective modules. The trivially fibrant objects are the categorically injective modules.

\item There is a cofibrantly generated abelian model structure on $R$-Mod in which every object is fibrant and the cofibrant objects are the Gorenstein $(\class{L},\class{A})$-projectives. The trivially cofibrant objects are the categorically projective modules.

\item Assume the class of Gorenstein $(\class{L},\class{A})$-flat modules is closed under extensions. Then there is a cofibrantly generated abelian ``mixed'' model structure on $R$-Mod whose cofibrant objects are the Gorenstein $(\class{L},\class{A})$-flat modules and the trivially cofibrant objects are the categorically flat modules. The fibrant objects are the cotorsion modules.  
\end{enumerate}
\end{theorem}

We point out that the homotopy categories of these model structures are well generated triangulated categories, in the sense of Neeman~\cite{neeman-well generated}. This is because they are \emph{hereditary} abelian model structures (so the homotopy categories are stable) and they are \emph{cofibrantly generated} (so the homotopy categories are well generated). 

We also point out that since the trivial objects in the projective and flat model structures coincide, they give two different models, in general, for the same stable module category. The class of trivial objects in this case contains all $R$-modules in $\class{L}$, so in particular all projective $R$-modules, but also all injective $R$-modules are trivial. On the other hand, for the model structure in (1), the class of trivial objects contains all $R$-modules in $\class{A}$, so in particular all injective $R$-modules, but also all projective $R$-modules are trivial.

This paper relies heavily on the work of~\cite{bravo-gillespie-hovey}. Although we give a few shortened proofs, the focus here is on the amazing amount of homotopy theory that follows automatically from the existence of a complete duality pair.

We refer the reader to~\cite{enochs-jenda-book} for the language of cotorsion pairs we will use, and to~\cite{hovey} and~\cite{gillespie-hereditary-abelian-models} for the interaction between cotorsion pairs and abelian model strucures.

\

\noindent \textbf{Acknowledgement:} The author wishes to thank Marco P\'erez for providing comments and feedback on an early draft of this paper.

\section{Complete duality pairs}\label{Sec-duality pairs}  

In this section we recall the original definition of duality pair, due to Holm and J\o rgensen. We then define symmetric duality pairs and complete duality pairs in order to merge their notions with the one from~\cite{bravo-gillespie-hovey}. We also list several interesting examples of duality pairs for which the results in this paper apply.

First, recall that for a given $R$-module $M$, its \emph{character module} is defined to be the $R$-module $M^+ = \Hom_{\Z}(M,\Q)$. 

\begin{definition}\cite[Definition~2.1]{holm-jorgensen-duality}\label{def-duality pair}
A \textbf{duality pair} over $R$ is a pair $(\class{M},\class{C})$, of classes of $R$-modules, satisfying (i) $M \in \class{M}$ if and only if $M^+ \in \class{C}$, and (ii) $\class{C}$ is closed under direct summands and finite direct sums. A duality pair $(\class{M},\class{C})$ is called \textbf{perfect} if $\class{M}$  contains the module $R$, and is closed under coproducts and extensions. 
\end{definition}

\begin{theorem}\cite[Theorem~3.1]{holm-jorgensen-duality}\label{them-duality pair purity}
Let $(\class{M},\class{C})$ be a duality pair. Then the following hold: 
\begin{enumerate}
\item $\class{M}$ is closed under pure submodules, pure quotients, and pure extensions. 
\item If $(\class{M},\class{C})$ is perfect, then $(\class{M}, \rightperp{\class{M}})$ is a perfect cotorsion pair.  
\end{enumerate}
\end{theorem}

The following observation will be important for this paper. 

\begin{proposition}\label{prop-flats}
If $(\class{M},\class{C})$ is a perfect duality pair, then $\class{M}$ contains all projective modules, in fact, it contains all flat modules. And, the class $\class{C}$ contains all injective modules.   
\end{proposition}

\begin{proof}
Since $\class{M}$ contains $R$ and is closed under coproducts, $\class{M}$ contains all free modules. By Theorem~\ref{them-duality pair purity}(1) $\class{M}$ is closed under direct summands (since direct summands are certainly pure submodules). So $\class{M}$ contains all projective modules. 
 
Now given a directed system of modules $\{M_i\}$, it is well known that the canonical morphism $\oplus M_i \twoheadrightarrow \varinjlim M_i$ is a pure epimorphism. Now Lazard's Theorem states that for any flat module $F$ we have $F \cong \varinjlim P_i$ for some directed system $\{P_i\}$ of (finitely generated) projectives $P_i$. So it follows again from Theorem~\ref{them-duality pair purity}(1) that $\class{M}$ contains all flat modules. 

We now show that $\class{C}$ contains all injective modules. We first note that since $\class{M}$ contains all free modules we have $\prod_{i\in I} R^+ \in \class{C}$ since $$\prod_{i\in I} \Hom_{\Z}(R,\Q) \cong \Hom_{\Z}(\bigoplus_{i\in I}R,\Q) = [\bigoplus_{i\in I}R]^+.$$ But $R^+$ is an injective cogenerator for $R$-Mod; for example, see~\cite[Def.~3.2.7]{enochs-jenda-book} or~\cite[Prop.~I.9.3]{stenstrom}. It means any injective module must be a direct summand of a module of the form $\prod_{i\in I} R^+$; for example, see~\cite[Prop.~I.6.6]{stenstrom}. So any injective must also be in $\class{C}$.
\end{proof}

The next definition essentially encapsulates the notion of duality pair used in~\cite{bravo-gillespie-hovey}.

\begin{definition}~\cite[Appendix~A]{bravo-gillespie-hovey}\label{def-complete duality pair}
By a \textbf{symmetric duality pair} over $R$ we mean a pair of classes $\{\class{L},\class{A}\}$ for which both $(\class{L},\class{A})$ and $(\class{A},\class{L})$ are duality pairs. By a \textbf{complete duality pair} $(\class{L},\class{A})$ we mean that $\{\class{L},\class{A}\}$ is a symmetric duality pair with $(\class{L},\class{A})$ being a perfect duality pair. In this case, because of Proposition~\ref{prop-flats}, we will call $\class{L}$ the \textbf{projective class} and $\class{A}$ the \textbf{injective class}. 
\end{definition}

Now the following is the main result concerning symmetric duality pairs.

\begin{theorem}\cite[Theorem~A.6]{bravo-gillespie-hovey}\label{them-projectivecomplexes}
Let $\{\class{L}, \class{A}\}$ be a symmetric duality pair and let $P$ denote a chain complex of projective $R$-modules. Then the following hold:
\begin{enumerate}
\item $\Hom_R(A,P)$ is exact for all $A \in \class{A}$ if and only if $\Hom_R(P,L)$ is exact for all $L \in \class{L}$.
\item $\Hom_R(L,P)$ is exact for all $L \in \class{L}$ if and only if $\Hom_R(P,A)$ is exact for all $A \in \class{A}$.
\end{enumerate}  
\end{theorem}

\subsection{Examples of complete duality pairs}\label{sec-example duality pairs}
Serveral classes of examples of duality pairs are given throughout~\cite{holm-jorgensen-duality, bravo-gillespie-hovey, bravo-perez}. We summarize here all of those that are complete duality pairs.

\begin{example}\label{example-level}
The canonical example of a complete duality pair is the level duality pair, $(\class{L},\class{A})$, given in~\cite{bravo-gillespie-hovey}. It exists over any ring $R$. To describe it, a module $F$ is said to be of \emph{type $FP_{\infty}$} if it has a projective resolution 
$$\cdots \xrightarrow{} P_2  \xrightarrow{} P_1  \xrightarrow{} P_0  \xrightarrow{} F  \xrightarrow{} 0$$ with each $P_i$ finitely generated. Then we make the following definitions. 
We then call an $R$-module $A$ \textbf{absolutely clean} if $\Ext^1_{R}(F,A)=0$ for all $R$-modules $F$ of type $FP_{\infty}$. On the other hand, we call an $R$-module $L$ \textbf{level} if $\Tor_1^R(F,L) = 0$ for all $R$-modules $F$ of type $FP_{\infty}$.
Letting $\class{L}$ denote the class of all level modules, and $\class{A}$ the class of all absolutely clean modules, we have a complete duality pair $(\class{L},\class{A})$ by~\cite[Section~2]{bravo-gillespie-hovey}. 

In fact, Bravo and P\'erez give a generalization of this in~\cite{bravo-perez}.
A module $F$ is said to be of \emph{type $FP_{n}$} (including possibly $n = \infty$) if it has a projective resolution 
$$P_n \xrightarrow{} \cdots \xrightarrow{} P_2  \xrightarrow{} P_1  \xrightarrow{} P_0  \xrightarrow{} F  \xrightarrow{} 0$$ with each $P_i$ finitely generated. Then an $R$-module $A$ is called \textbf{$\boldsymbol{\text{FP}_n}$-injective} if $\Ext^1_{R}(F,A)=0$ for all $R$-modules $F$ of type $FP_{n}$. On the other hand, we can $R$-module $L$ is called \textbf{$\boldsymbol{\text{FP}_n}$-flat} if $\Tor_1^R(F,L) = 0$ for all $R$-modules $F$ of type $FP_{n}$.
Let $\class{FP}_n\text{-Flat}$ denote the class of all $\text{FP}_n$-flat modules, and $\class{FP}_n\text{-Inj}$ denote the class of all $\text{FP}_n$-injective modules. Then for all $n \geq 2$, we have a complete duality pair $(\class{FP}_n\text{-Flat},\class{FP}_n\text{-Inj})$ by~\cite[Cor.~3.7]{bravo-gillespie-hovey}.
\end{example}

The remaining sequence of examples come from~\cite{holm-jorgensen-duality}. Please see that paper for further detail, references, and unexplained terminology. 

\begin{example}\label{example-auslander}
Let $R$ be a commutative Noetherian ring with a semidualizing $R$-complex $C$. There is associated, the so-called \emph{auslander class}, denoted $\class{A}^C_0$, and \emph{Bass class}, denoted $\class{B}^C_0$. Then  $(\class{A}^C_0, \class{B}^C_0)$ is a complete duality pair by~\cite[Prop.~2.4]{holm-jorgensen-duality}. 
\end{example}

\begin{example}\label{example-C-dimension}
Let $R$ be a commutative Noetherian ring with a semidualizing $R$-module $C$. Holm and J\o rgensen introduced in~\cite{holm-jorgensen-C-G-flat dimensions} the \emph{$C$-Gorenstein flat dimension}, denoted $\textnormal{C-Gfd}_RM$, and the \emph{$C$-Gorenstein injective dimension}, denoted $\textnormal{C-Gid}_RM$, of an $R$-module $M$. They define the following two classes:
$$\class{GF}^C_n = \{\, M \in R\textnormal{-Mod} \,|\, \textnormal{C-Gfd}_RM \leq n  \,\} $$
$$\class{GI}^C_n = \{\, M \in R\textnormal{-Mod} \,|\, \textnormal{C-Gid}_RM \leq n  \,\}. $$
(For example, $\class{GF}^R_0$ is the class of all usual Gorenstein flat $R$-modules.) 
It is shown in~\cite[Prop.~2.6]{holm-jorgensen-duality} that each $(\class{GF}^C_n, \class{GI}^C_n)$ is a complete duality pair whenever the ring $R$ admits a dualizing complex. 
\end{example}

\begin{example}\label{example-width}
Let $(R, \mathfrak{m}, k)$ be a commutative Noetherian local ring. For a given $R$-module $M$, Foxby defined its \emph{depth} in~\cite{foxby-bounded flats}, and later Yassemi defined the dual \emph{width} in~\cite{yassemi-width}. They are defined as follows:
$$\textnormal{depth}_RM = \textnormal{inf}\{\, m \in \Z \,|\, \Ext^m_R(k,M) \neq 0 \,\} $$
$$\textnormal{width}_RM = \textnormal{inf}\{\, m \in \Z \,|\, \Tor_m^R(k,M) \neq 0 \,\}. $$
Now, for each $n \geq 0$, define the following classes:
$$\class{D}_n = \{\, M \in R\textnormal{-Mod} \,|\, \textnormal{depth}_RM \geq n \,\} $$
$$\class{W}_n = \{\, M \in R\textnormal{-Mod} \,|\, \textnormal{width}_RM \geq n \,\}. $$
It is shown in~\cite[Prop.~2.7]{holm-jorgensen-duality} that each $(\class{D}_n, \class{W}_n)$ satisfying $\textnormal{depth}R \geq n$ is a complete duality pair. 

\end{example}

\section{Duality pairs and acyclic model structures}\label{Sec-acyclic models}

In this section we show how a complete duality pair $(\class{L},\class{A})$ over $R$ induces two abelian model structures on $\ch$. One is an injective model structure and turns out to be Quillen equivalent to the Gorenstein $(\class{L},\class{A})$-injective model structure of Section~\ref{sec-Gorenstein}. The other is a projective model structure and is Quillen equivalent to the Gorenstein $(\class{L},\class{A})$-projective model structure. The main results here are
Theorems~\ref{theorem-injective model on complexes} and~\ref{theorem-projective model on complexes} and follow from work in~\cite{bravo-gillespie-hovey}.

\begin{corollary}\label{transfinite}
Suppose $\class{L}$ and $\class{A}$ form a symmetric duality pair over $R$. Then there exists a set (not just a proper class) $S \subseteq \class{L}$  such that each $L \in \class{L}$ is a transfinite extension of modules in $S$. Similarly there is a set $T \subseteq \class{A}$ such that each $A \in \class{A}$ is a transfinite extension of modules in $T$. 
\end{corollary} 

\begin{proof}
By Proposition~\ref{them-duality pair purity}(1), each class is closed under pure submodule and pure quotients. So the statement follows from~\cite[Propositioin~2.8]{bravo-gillespie-hovey}.
\end{proof}



Let us now recall some definitions from~\cite{bravo-gillespie-hovey}. 

\begin{definition}\label{def-acyclicity}
Given an $R$-module $M$, a chain complex $X$ of injective $R$-modules is called \textbf{$M$-acyclic} if $\Hom_R(M,X)$ is exact. In a similar way, given a class $\class{M}$ of $R$-modules, we will call $X$ \emph{$\class{M}$-acyclic} if it is $M$-acyclic for all $M \in \class{M}$. On the other hand, if $X$ is a chain complex of projective (or even flat) $R$-modules, we call it \textbf{$M$-acyclic} if $M \otimes_R X$ is exact; and similarly we define $\class{M}$-acyclic complexes of projectives for a class $\class{M}$. In any case, we note that we are not necessarily assuming that $X$ is itself exact. We will  specify this explicitly by using terminology such \emph{$X$ is exact $M$-acyclic}.
\end{definition}

Note that if the class $\class{M}$ contains $R$, then $\class{M}$-acyclic complexes are certainly always exact (acyclic in the usual sense). 

\begin{lemma}\label{test modules}
Let $\class{M}$ be a class of $R$-modules for which there exists a set (so again not just a class) $S \subseteq \class{M}$ such that each $M \in \class{M}$ is a transfinite extension of modules in $S$. Set $M' = \bigoplus_{N \in S} N$, and set $M = R \oplus M'$. Then the following hold for any chain complex $X$ of injective (resp. projective) $R$-modules. 
\begin{enumerate}
\item $X$ is $\class{M}$-acyclic if and only if it is $M'$-acyclic.
\item $X$ is exact $\class{M}$-acyclic if and only if it is $M$-acyclic.
\end{enumerate}
\end{lemma}

\begin{proof}
First let $X$ be a complex of injectives. If $X$ is $\class{M}$-acyclic, then in particular $\Hom_R(N,X)$ is exact for each $N \in S$, and hence $$\prod_{N \in S} \Hom_R(N,X) \cong \Hom_R(M',X)$$ is exact. This means $\class{M}$-acyclic implies $M'$-acyclic. And since $\Hom_R(R,X) \cong X$, we see similarly that $X$ exact $\class{M}$-acyclic implies $X$ is $M$-acyclic. 

Conversely, if $X$ is $M'$-acyclic, then the isomorphism above implies that $X$ is $N$-acyclic for each $N \in \class{S}$ (and that $X$ is itself exact in the case that $X$ is $M$-acyclic). It is left to show that $X$ is $M$-acyclic for any given $M \in \class{M}$. But since $X$ is a complex of injectives we have (as in the proof of~\cite[Theorem~4.1]{bravo-gillespie-hovey}) that
\[
H_{n-1}\Hom_R(M,X) \cong \Ext^{1}_{\ch} (S^{n}(M),X).
\] 
In particular, $X$ is $M$-acyclic if and only if $ \Ext^{1}_{\ch} (S^{n}(M),X) = 0$ for each $n$. But each $M \in \class{M}$ is a transfinite extension of modules in $S$. So the desired conclusion now follows from a well known result known as Eklof's lemma. 

The case with $X$ a complex of projectives is similar. In fact it is easier to prove because transfinite extensions are particular types of directed colimits and tensor products commutes with direct limits.  
\end{proof}

\begin{theorem}[$\class{A}$-acyclic injective model structures]\label{theorem-injective model on complexes}
Let $(\class{L},\class{A})$ be a complete duality pair over $R$. Then there is a cofibrantly generated abelian model structure on $\ch$ whose fibrant objects are the exact $\class{A}$-acyclic complexes of injectives. The model structure is injective in the sense that every object is cofibrant and the trivially fibrant objects are the categorically injective chain complexes. 

There is a similar model structure whose fibrant objects are the (not necessarily exact) $\class{A}$-acyclic complexes of injectives. Moreover, there is yet another similar model structure whose fibrant objects are the (necessarily) exact $\class{L}$-acyclic complexes of injectives.  
\end{theorem}

The first model structure is highlighted in the theorem because it is the important one that appears later as Quillen equivalent to a model structure on $R$-Mod. 

\begin{proof}
By Corollary~\ref{transfinite} we know that there exists a set $T \subseteq \class{A}$ so that each $A \in \class{A}$ is a transfinite extension of modules in $T$. So the $\class{A}$-acyclic model structures exist by~\cite[Theorem~4.1]{bravo-gillespie-hovey} combined with the above Lemma~\ref{test modules}. Of course we can do the same thing for a set $S \subseteq \class{L}$. But note $\class{L}$-acyclic complexes are necessarily exact since $R \in \class{L}$. 
\end{proof}

Work already done in~\cite{bravo-gillespie-hovey} also provides the dual model structures:

\begin{theorem}[$\class{A}$-acyclic projective model structures]\label{theorem-projective model on complexes}
Let $(\class{L},\class{A})$ be a complete duality pair over $R$. Then there is a cofibrantly generated abelian model structure on $\ch$ whose cofibrant objects are the exact $\class{A}$-acyclic complexes of projectives. The model structure is projective in the sense that every object is fibrant and the trivially cofibrant objects are the categorically projective chain complexes. 

There is a similar model structure whose cofibrant objects are the (not necessarily exact) $\class{A}$-acyclic complexes of projectives. Moreover, there is yet another similar model structure whose cofibrant objects are the (necessarily) exact $\class{L}$-acyclic complexes of projectives.  
\end{theorem}

Again, it is the first model structure highlighted in the theorem that will be the interesting one in this paper. 

\begin{proof}
Again, by Corollary~\ref{transfinite} we know that there exists a set $T \subseteq \class{A}$ so that each $A \in \class{A}$ is a transfinite extension of modules in $T$. This time the $\class{A}$-acyclic model structures exist by~\cite[Theorem~6.1]{bravo-gillespie-hovey} combined with the above Lemma~\ref{test modules}. When we do the same thing for a set $S \subseteq \class{L}$ we get the exact $\class{L}$-acyclic model structure. We note that $\class{L}$-acyclic complexes of projectives are necessarily exact because $R \in \class{L}$ and $R \tensor_R X \cong X$. 
\end{proof}

\section{Gorenstein homological algebra relative to a duality pair}\label{sec-Gorenstein}
 
The goal of this Section is to show how Gorenstein homological algebra can be done with respect to a duality pair. It comes from generalizing the core idea behind the main results of~\cite{bravo-gillespie-hovey}. 

\begin{setup}
Throughout this section we let $(\class{L},\class{A})$ denote a fixed complete duality pair over $R$.
\end{setup}

See Definition~\ref{def-acyclicity} to recall our terminology for acyclic complexes of injectives and projectives. Also, given a class $\class{C}$ of $R$-modules, we define $\leftperp{\class{C}}$ to be the class of all $R$-modules $M$ such that $\Ext^1_R(M,C) = 0$ for all $C \in \class{C}$. The class $\rightperp{\class{C}}$ is defined similarly. 

\begin{definition}\label{def-G-inj}
An $R$-module $M$ is called \textbf{Gorenstein $(\class{L},\class{A})$-injective} if
$M=Z_{0}I$ for some exact $\class{A}$-acyclic complex of injectives $I$ .  We let $\class{GI}$ denote the class of all Gorenstein $(\class{L},\class{A})$-injective modules and set $\class{W} = \leftperp{\class{GI}}$.
\end{definition}

\begin{definition}\label{def-G-proj}
An $R$-module $M$ is called \textbf{Gorenstein $(\class{L},\class{A})$-projective} if
$M=Z_{0}P$ for some exact $\class{A}$-acyclic complex of projectives $P$. By Theorem~\ref{them-projectivecomplexes}, it is equivalent to require that the complex of projectives $P$ remains exact after applying $\Hom_R(-,L)$ for any $L \in \class{L}$. We let $\class{GP}$ denote the class of all Gorenstein $(\class{L},\class{A})$-projective modules and set $\class{V} = \rightperp{\class{GP}}$.    
\end{definition}

\begin{remark}
Let $(\class{L},\class{A})$ denote the level duality pair of Example~\ref{example-level}.
In the case that $R$ is Noetherian the class $\class{A}$ coincides with the class of injective modules. So Gorenstein $(\class{L},\class{A})$-injective coincides with the usual notion of Gorenstein injective. Similarly, the notion of Gorenstein $(\class{L},\class{A})$-flat given in Definition~\ref{def-G-flat} coincides with the usual notion of Gorenstein flat. If furthermore, all flat modules have finite projective dimension, for example $R$ commutative of finite Krull dimension, then one can argue that the Gorenstein $(\class{L},\class{A})$-projectives coincide with the usual Gorenstein projectives too. 
\end{remark}

We now focus on proving results concerning the Gorenstein $(\class{L},\class{A})$-injectives. In particular, our goal is to show that $(\class{W},\class{GI})$ is an injective cotorsion pair and that its associated model structure on $R$-Mod is Quillen equivalent to the exact $\class{A}$-acyclic model structure of Theorem~\ref{theorem-injective model on complexes}. We will first need some lemmas relating $\class{GI}$ and $\class{W}$ to the that model structure on $\ch$. We follow the approach from~\cite{bravo-gillespie-hovey} but with some different proofs.

\begin{lemma}\label{lem-spheres}
$W \in \class{W}$ if and only
if $S^{0}(W)$ is trivial in the exact $\class{A}$-acyclic model structure.
\end{lemma}

\begin{proof}
It follows from the fact that for any exact complex $X$, we have an isomorphism $\Ext^{1}_{\ch} (S^{0}(W), X) \cong \Ext^{1}_R(W, Z_{0}X)$; see~\cite[Lemma~4.2]{gillespie-degreewise-model-strucs}.
\end{proof}

\begin{lemma}\label{lem-cycles-of-W}
Suppose $X$ is a complex with $H_{i}X=0$
for $i<0$ and $X_{i} \in \class{A}$ for $i>0$.  Then $X$ is trivial in the
exact $\class{A}$-acyclic model structure if and only if $Z_{0}X \in \class{W}$.
\end{lemma}

\begin{proof}
This is the analog of~\cite[Lemma~5.1]{bravo-gillespie-hovey}, but we sketch an alternate proof. First note that the proof of the exact $\class{A}$-acyclic model structure on $\ch$ reveals that the cotorsion pair is cogenerated by the set $\{D^n(R/\mathfrak{a}) , S^n(A_i) \}$ where $\mathfrak{a}$ ranges through all (left) ideals of $R$ and $A_i$ ranges through a set $T \subseteq \class{A}$ as in Corollary~\ref{transfinite}. So the trivial complexes are precisely retracts of transfinite extensions of such sphere and disk complexes. In particular it implies (i) any disk $D^n(M)$ on any module $M$ is trivial, and (ii) any sphere $S^n(A)$, where $A \in \class{A}$, is trivial. From (i) we infer that any bounded above exact complex is trivial, as any such complex is a transfinite extension of disks. From (ii) we infer that any bounded below complex of modules in $\class{A}$ is trivial, as any such complex is a transfinite extension of spheres $S^n(A)$.  Now the chain complex $X$ given has a subcomplex $B \subseteq X$, where $B$ is the bounded below complex $\cdots \xrightarrow{} X_2 \xrightarrow{} X_1 \xrightarrow{} Z_0X \xrightarrow{} 0$. Then note that $X/B$ is (isomorphic to) the complex $0 \xrightarrow{} Z_{-1}X \xrightarrow{} X_{-1} \xrightarrow{} X_{-2} \xrightarrow{} \cdots$, which is bounded above and exact and so is trivial in the exact $\class{A}$-acyclic model structure. Therefore, since the class of trivial complexes is thick, $X$ is trivial if and only if $B$ is trivial. But we have another subcomplex $S^0(Z_0X) \subseteq B$, whose quotient is a bounded below complex of modules in $\class{A}$. Thus $B$ (and hence $X$) is trivial if and only if  $S^0(Z_0X)$ is trivial. By Lemma~\ref{lem-spheres}, this happens if and only if $Z_0X \in \class{W}$.
\end{proof}

\begin{lemma}\label{lemma-retracts}
Again let $\class{GI}$ denote the class of Gorenstein $(\class{L},\class{A})$-injectives.
\begin{enumerate}
\item $\class{GI}$ is closed under products.
\item $\class{GI}$ is injectively resolving in the sense of~\cite[Definition~1.1]{holm}.
\item $\class{GI}$ is closed under retracts.
\end{enumerate}
\end{lemma}

\begin{proof}
For (1), it is easy to see from the definition that $\class{GI}$ is also closed under direct products.

For (2), we note that Yang, Liu, and Liang show in~\cite[Theorem~2.7]{Ding projective} that the class of all Ding injective $R$-modules is injectively resolving in the sense of~\cite[Definition~1.1]{holm}. It means that the class contains the injectives, is closed under extensions, and is closed under taking cokernels of monomorphisms. Although the proof they give is for ``Ding'' $R$-modules (and is the projective version), the elegant arguments hold in the same exact way to show $\class{GI}$ is injectively resolving.

Holm shows in~\cite[Proposition~1.4]{holm} how an Eilenberg swindle argument can be used to conclude (3) from both (1) and (2).
\end{proof}

\begin{theorem}\label{thm-Gor-module}
There is an abelian model structure on $R$-Mod, the \textbf{Gorenstein $(\class{L},\class{A})$-injective model structure}, in which every object is cofibrant and the fibrant objects are the Gorenstein $(\class{L},\class{A})$-injectives.
\end{theorem}

\begin{proof}
Again, we have taken $\class{GI}$ to be the Gorenstein $(\class{L},\class{A})$-injective objects,
and define $\class{W}=\leftperp{\class{GI}}$.  Then
Lemma~\ref{lem-spheres} shows that $\class{W}$ is thick and contains
the injectives since $\class{A}$ contains the injectives.  Now for any object $M$, we have a short exact sequence
\[
0 \xrightarrow{} S^{0}M \xrightarrow{} I \xrightarrow{}X \xrightarrow{} 0
\]
in which $I$ is an exact $\class{A}$-acyclic complex of injectives and $X$ is
trivial in the exact $\class{A}$-acyclic model category.  By the snake lemma, we get a short exact sequence
\[
0 \xrightarrow{} M \xrightarrow{} Z_{0}I \xrightarrow{} Z_{0}X \xrightarrow{} 0
\]
Of course $Z_{0}I$ is Gorenstein $(\class{L},\class{A})$-injective by definition, but $Z_{0}X$
is in $\class{W}$ as well by Lemma~\ref{lem-cycles-of-W}, since
$X_{i} \in \class{A}$ is injective for all $i\neq 0$ and $H_{i}X=0$ for all $i\neq
1$.  So the purported cotorsion pair $(\class{W},\class{GI})$ has
enough injectives, if it is a
cotorsion pair. The standard pushout argument of Salce will then apply to show that the purported cotorsion pair also has enough projectives, and so is a complete cotorsion pair.

But of course it is left show that $\class{GI} \supseteq \rightperp{\class{W}}$,
so that we know $(\class{W},\class{GI})$ is actually a cotorsion pair.
So suppose $M \in \rightperp{\class{W}}$. We can now find a short exact
sequence
\[
0 \xrightarrow{} M \xrightarrow{} J \xrightarrow{} W \xrightarrow{} 0
\]
where $J$ is Gorenstein $(\class{L},\class{A})$-injective and $W \in \class{W}$.  By
assumption, this must split, and so $M$ is a retract of $J$, and so Gorenstein $(\class{L},\class{A})$-injectives too by Lemma~\ref{lemma-retracts}(3).
\end{proof}

Following the same exact method as from~\cite[Prop.~5.10]{bravo-gillespie-hovey}, one obtains the following corollary.

\begin{corollary}\label{cor-cogenerated}
The cotorsion pair $(\cat{W}, \cat{GI})$, where
$\cat{GI}$ is the class of Gorenstein $(\class{L},\class{A})$-injectives, is cogenerated by
a set. Thus the Gorenstein $(\class{L},\class{A})$-injective model structure is
cofibrantly generated.
\end{corollary}

Furthermore, the ``same'' proof as~\cite[Theorem~5.8]{bravo-gillespie-hovey} holds and gives the following corollary.

\begin{corollary}\label{cor-quillen equiv}
The functor $S^0(-) : R\text{-Mod} \xrightarrow{} \ch$ is a Quillen equivalence from the Gorensein $(\class{L},\class{A})$-injective model structure to the exact $\class{A}$-acyclic injective model structure. 
\end{corollary}

\begin{remark}
Each of the above results from Lemma~\ref{lem-spheres} through Corollary~\ref{cor-quillen equiv} has a corresponding projective version. The proofs are all similar; for some proofs one must switch to the second characterization of Gorenstein $(\class{L},\class{A})$-projective given in Definition~\ref{def-G-proj}.  
\end{remark}

We summarize the main results of the dual situation in the following theorem. 

\begin{theorem}\label{thm-Gor-proj-module}
There is a cofibrantly generated abelian model structure on $R$-Mod, the \textbf{Gorenstein $(\class{L},\class{A})$-projective model structure}, in which every object is fibrant and the cofibrant objects are the Gorenstein $(\class{L},\class{A})$-projectives.

The functor $S^0(-) : R\text{-Mod} \xrightarrow{} \ch$ is a (right) Quillen equivalence from the Gorensein $(\class{L},\class{A})$-projective model structure to the exact $\class{A}$-acyclic projective model structure. 
\end{theorem}

\section{Gorenstein $(\class{L},\class{A})$-Flat modules}

We now look at the analog to Gorenstein flat, but again relative to a complete duality pair $(\class{L},\class{A})$.  
We construct a model structure which, when it exists, has the same trivial objects as those in Theorem~\ref{thm-Gor-proj-module}, but with cofibrant objects the Gorenstein $(\class{L},\class{A})$-flat modules. We again will work with the following set up. 

\begin{setup}
Throughout this section we let $(\class{L},\class{A})$ denote a fixed complete duality pair over $R$.
\end{setup}

Remembering yet again the conventions set in Definition~\ref{def-acyclicity}, we make the following definition. 

\begin{definition}\label{def-G-flat}
An $R$-module $M$ is called \textbf{Gorenstein $(\class{L},\class{A})$-flat} if
$M=Z_{0}F$ for some exact $\class{A}$-acyclic complex of flat modules $F$. We let $\class{GF}$ denote the class of all Gorenstein $(\class{L},\class{A})$-flat modules and denote $\class{GC} = \rightperp{\class{GF}}$. Modules in $\class{GC}$ are called \textbf{Gorenstein $(\class{L},\class{A})$-cotorsion} modules.
\end{definition}

We won't prove the following proposition here because it follows at once by a recent result of Estrada, Iacob, and P\'erez.

\begin{proposition}~\cite[Corollary~2.7]{estrada-iacob-perez-G-flat}\label{prop-G-flat}
Assume the class $\class{GF}$ is closed under extensions. Then $(\class{GF},\class{GC})$ is a complete hereditary cotorsion pair, and cogenerated by a set.  
\end{proposition}

\begin{proof}
Take the class $\class{A}$ from our complete duality pair $(\class{L},\class{A})$. Use it when applying ~\cite[Cor.~2.7]{estrada-iacob-perez-G-flat} to obtain the result.
\end{proof}

We also will need the following deep result of Amnon Neeman.

\begin{proposition}~\cite[Theorem~8.6]{neeman-flat}\label{lemma-Neeman}
Let $R$ be any ring, $\dwclass{P}$ (resp. $\dwclass{F}$) denote the class of all complexes of projectives (resp. flats), and $\tilclass{F}$ denote the class of all categorically flat complexes; that is, exact complexes with all flat cycles. Then $\dwclass{F} \cap \rightperp{(\dwclass{P})} = \tilclass{F}$.
\end{proposition}

Now let $(\class{F},\class{C})$ denote Enochs' \emph{flat cotorsoin pair}. So $\class{F}$ is the class of all flat modules and $\class{C} = \rightperp{\class{F}}$ is the class of \emph{cotorsion modules}. These will be the fibrant objects in the model structure presented in the next theorem.

\begin{theorem}\label{them-trivial-objects}
Assume the class $\class{GF}$ of Gorenstein $(\class{L},\class{A})$-flat modules is closed under extensions. Then $(\class{GF},\class{W},\class{C})$ is a cofibrantly generated Hovey triple and $\class{W} = \rightperp{\class{GP}}$ are precisely the same trivial objects as in the Gorenstein $(\class{L},\class{A})$-projective model structure of Theorem~\ref{thm-Gor-proj-module}.
\end{theorem}

\begin{proof}
We have two complete hereditary cotorsion pairs $(\class{F},\class{C})$ and $(\class{GF},\class{GC})$, each cogenerated by a set. 
It is immediate from the definitions that $\class{GP} \subseteq \class{GF}$, whence $\class{GC} \subseteq \class{W}  = \rightperp{\class{GP}}$. Also, $\class{W}$ is already known to be thick by (the dual of) Lemma~\ref{lem-spheres}. So the theorem will follow immediately from~\cite[Lemma~2.3(1)]{gillespie-models-for-hocats-of-injectives} once we show $\class{GF} \cap \class{W} = \class{F}$, where $\class{F}$ is the class of all flat modules.

Now since $\class{L}$ contains the flat modules (Proposition~\ref{prop-flats}) it follows from the alternate characterization given in Definition~\ref{def-G-proj} that $\class{W} \supseteq \class{F}$. Hence we see that $\class{GF} \cap \class{W} \supseteq \class{F}$, so all that remains is to show
$\class{GF} \cap \class{W} \subseteq \class{F}$. 

So let $M \in \class{GF} \cap \rightperp{\class{GP}}$, and write it as $M = Z_0F$ where $F$ is an exact $\class{A}$-acyclic complex of flat modules; so $A \tensor_R F$ remains exact for all $A \in \class{A}$. From~\cite[Cor.~6.4]{bravo-gillespie-hovey} we have a complete cotorsion pair $(\dwclass{P}, \rightperp{(\dwclass{P})})$, where $\dwclass{P}$ is the class of all complexes of projectives. So we may write a short exact sequence $$0 \xrightarrow{} F  \xrightarrow{} W  \xrightarrow{} P  \xrightarrow{} 0 $$ with $W \in  \rightperp{(\dwclass{P})}$ and $P \in \dwclass{P}$. But then using Proposition~\ref{lemma-Neeman}, one easily argues that $W \in \tilclass{F}$, the class of all exact complexes with all cycle modules flat. Since $F$ and $W$ are each exact, we see that $P$ is exact too. Moreover,  the short exact sequence is split in each degree and so for any $A \in \class{A}$ we have a short exact sequence   $$0 \xrightarrow{} A \tensor_R F  \xrightarrow{} A \tensor_R W  \xrightarrow{} A \tensor_R P  \xrightarrow{} 0. $$ It is now clear that $A \tensor_R P$ is also exact, equivalently, $\Hom_R(P,L)$ is exact for all $L \in \class{L}$, and so $Z_0P$ is Gorenstein $(\class{L},\class{A})$-projective. Note that since each complex is exact we have a short exact sequence $0 \xrightarrow{} Z_0F  \xrightarrow{} Z_0W  \xrightarrow{} Z_0P  \xrightarrow{} 0$. By the hypothesis, $Z_0F = M \in \rightperp{\class{GP}}$, and so we conclude  that this sequence splits. Since $Z_0W$ is flat, so is the direct summand $Z_0F$, proving $\class{GF} \cap \rightperp{\class{GP}} \subseteq \class{F}$.
\end{proof}

\end{document}